\documentclass{amsart}

\newtheorem{twr}{Theorem}[section]
\newtheorem{lem}[twr]{Lemma}
\newtheorem{conj}[twr]{Conjecture}
\newtheorem{coll}[twr]{Corollary}
\newtheorem{problem}[twr]{Problem}
\theoremstyle{remark}

\numberwithin{equation}{section}

\begin{document}

\title{Equilateral dimension of some classes of normed spaces}

\commby{}


\author[T. Kobos]{Tomasz Kobos}
\address{Faculty of Mathematics and Computer Science \\ Jagiellonian University \\ Lojasiewicza 6, 30-348 Krakow, Poland}
\email{Tomasz.Kobos@im.uj.edu.pl}



\begin{abstract}
The equilateral dimension of a normed space is the maximal number of pairwise equidistant points of this space. The aim of this paper is to study the equilateral dimension of certain classes of finite dimensional normed spaces. A well-known conjecture states that the equilateral dimension of any $n$-dimensional normed space is not less than $n+1$. By using an elementary continuity argument, we establish it in the following classes of spaces: permutation-invariant spaces, Musielak-Orlicz spaces and in one codimensional subspaces of $\ell^n_{\infty}$. 
For smooth and symmetric spaces, Musielak-Orlicz spaces satisfying an additional condition and every $(n-1)$-dimensional subspace of $\ell^{n}_{\infty}$ we also provide some weaker bounds on the equilateral dimension for every space which is sufficiently close to one of these. This generalizes a result of Swanepoel and Villa concerning the $\ell_p^n$ spaces.
\end{abstract}

\maketitle

\section{Introduction}

Let $X$ be a real $n$-dimensional vector space endowed with a norm $|| \cdot ||$. We say that a set $S \subset X$ is \emph{equilateral}, if there is a $p>0$ such that $||x-y||=p$ for all $x, y \in S, x \neq y$. By $e(X)$ let us denote the \emph{equilateral dimension} of the space $X$, defined as the maximal cardinality of an equilateral set in $X$. We will be concerned with lower bounds on equilateral dimension in certain classes of normed spaces. It is not hard to see that every equilateral set in $X$ corresponds to a family of pairwise touching translates of the unit ball of $X$. It  is widely conjectured (see e.g. \cite{grunbaum}, \cite{petty}, \cite{thompson}, \cite{makeev}, \cite{averkov}) that in any $n$-dimensional space $X$ we can find $n+1$ equidistant points, or equivalently, that every symmetric convex body in $\mathbb{R}^n$ has $n+1$ pairwise touching translates.

\begin{conj}
\label{dimension}
Let $X$ be an $n$-dimensional normed space. Then $e(X) \geq n+1$.
\end{conj}

This conjecture is proved for $n \leq 4$ (see \cite{makeev} and \cite{petty}) but surprisingly it remains open for all $n \geq 5$. There are some partial results known. Brass in \cite{brass} and Dekster in \cite{dekster}, following the same method, have independently found a general lower bound on the equilateral dimension that goes to infinity with the dimension going to infinity. To establish such a lower bound, they have used the Brouwer Fixed Point Theorem to prove that Conjecture \ref{dimension} holds in every space which is sufficiently close to the Euclidean space. The distance between $n$-dimensional normed spaces is measured by the so called (multiplicative) \emph{Banach-Mazur distance}, defined as $d(X, Y)= \inf ||T|| \cdot ||T||^{-1}$, where the infimum is taken over all linear, invertible operators $T: X \to Y$. They obtained

\begin{twr}[Brass \cite{brass} \& Dekster \cite{dekster}]
\label{brassdekster}
Let $X$ be an $n$-dimensional normed space with the Banach-Mazur distance $d(X,\ell_2^n)\leq
1+\frac{1}{n}$. Then an equilateral set in $X$ of at most $n$ points can be extended to an equilateral one of $n+1$ points. In particular, $e(X)\geq n+1$.
\end{twr}

Swanepoel and Villa in \cite{swanepoel2}, applying the method of Brass and Dekster, based on the Brouwer Fixed Point Theorem, have managed to improve their bound. In particular, they found an $\ell_{\infty}$ analogue of Theorem \ref{brassdekster}.
\begin{twr}[Swanepoel, Villa \cite{swanepoel2}]
\label{swanepoelvilla}
Let $X$ be an $n$-dimensional normed space with Banach-Mazur distance $d(X, \ell_\infty^n)\leq \frac{3}{2}$. Then $e(X)\geq n+1$.
\end{twr}
It turns out that the constant $\frac{3}{2}$ can be easily improved to $2$ as observed by Averkov in \cite{averkov}, who generalized this result for embeddings of more general finite metric spaces. In other words, Conjecture \ref{dimension} is true in every $n$-dimensional space with the Banach-Mazur distance to $\ell_{\infty}^n$ not greater than $2$. Swanepoel and Villa have pursued their method even further, proving that $n$-dimensional spaces which are sufficiently close to some of the $\ell_{p}^{n}$ spaces ``almost'' satisfy Conjecture \ref{dimension}. Specifically, we have the following

\begin{twr}[Swanepoel, Villa \cite{swanepoel2}]
\label{swanepoelvilla2} For each $n>2$ and $p \in (1,+\infty)$ there exists $R(p,n)>1$ such that for any $n$-dimensional normed space $X$ with Banach-Mazur distance $d(X,\ell_p^n)\le R(p,n)$ we have $e(X)\geq n$. Moreover,
$$R(p,n) = \max_{\theta>0} \left(\frac{1+(1+\theta)^p}{2+(n-2)\theta^p}\right)^{1/p} \sim 1+\frac{p-1}{2p}n^{-\frac{1}{p-1}}\text{ as $n\to\infty$ with $p$ fixed.}$$
\end{twr}

Even if Conjecture \ref{dimension} is believed to be true, spaces which satisfy the conditions of theorem of Brass-Dekster or Swanepoel-Villa, are the only examples existing in the literature for which we know that Conjecture \ref{dimension} holds. Our main goal is to provide some evidence for Conjecture \ref{dimension} by proving it in some other broad classes of normed spaces, or at least give some good lower bound on the equilateral dimension.

We start with a class of normed spaces defined by a geometric property. We say that an $n$-dimensional normed space $X=(\mathbb{R}^n, || \cdot ||)$ is a \emph{permutation-invariant} space if for every permutation $\sigma: \{1, 2, \ldots, n \} \to \{1, 2, \ldots, n \}$ we have 
$$||(x_1, x_2, \ldots, x_n)|| = ||(x_{\sigma(1)}, x_{\sigma(2)}, \ldots, x_{\sigma(n)})||.$$
In other words, the linear mapping $\mathbb{R}^{n} \ni (x_1, x_2, \ldots, x_n) \to (x_{\sigma(1)}, x_{\sigma(2)}, \ldots, x_{\sigma(n)})$ is an isometry of $X$. Or geometrically, the unit ball of $X$ is symmetric with respect to every hyperplane of the form 
$$\{(x_1, x_2, \ldots, x_n) \in \mathbb{R}^n: x_i=x_{i+1} \}.$$
For the permutation-invariant spaces we have
\begin{twr}
\label{perm}
Let $X$ be an $n$-dimensional permutation-invariant normed space. Then $e(X) \geq n+1$.
\end{twr}

An $n$-dimensional normed space $X=(\mathbb{R}^n, || \cdot ||)$ is called \emph{$1$-unconditional} (or \emph{absolute} following \cite{matrix}) if for every choice of signs $(\varepsilon_1, \varepsilon_2, \ldots, \varepsilon_n) \in \{-1, 1\}^n$ we have 
$$||(x_1, x_2, \ldots, x_n)|| = ||(\varepsilon_1 x_1, \varepsilon_2 x_2, \ldots, \varepsilon_n x_n)||.$$
Similarly like before, it is equivalent to the fact that the linear mapping $\mathbb{R}^{n} \ni (x_1, x_2, \ldots, x_n) \to (\varepsilon_1 x_1, \varepsilon_2 x_2, \ldots, \varepsilon_n x_n)$ is an isometry of $X$, or geometrically, the unit ball of $X$ is symmetric with respect to every hyperplane of the form 
$$\{(x_1, x_2, \ldots, x_n) \in \mathbb{R}^n: x_i=0 \}.$$

An $n$-dimensional normed space $X$ is called \emph{symmetric} if it is both permutation-invariant and $1$-unconditional. Moreover, $X$ is called \emph{smooth} if the unit ball has exactly one supporting functional at each point of the unit sphere (recall that if $||x_0||=1$ then $f$ is supporting functional at $x_0$ if $f(x_0)=1$ and $||x|| \leq 1 \Rightarrow |f(x)| \leq 1$). It turns out, that for symmetric spaces which are additionally smooth, we can find a non-trivial neighbourhood (in the sense of Banach-Mazur distance) in which every space has equilateral dimension not less than $n$. In other words, we provide the following generalization of Theorem \ref{swanepoelvilla2}.

\begin{twr}
\label{symm}
Let $X$ be a smooth and symmetric $n$-dimensional normed space. Then, there exists $R(X)>1$ such that $e(Y) \geq n$ for every normed space $Y$ satisfying $d(X, Y) \leq R(X)$. Moreover, $R(X) \geq 1 + \frac{\varepsilon_0}{6n}$, where $\varepsilon_0>0$ satisfies $\frac{\rho_X(\varepsilon_0)}{\varepsilon_0} \leq \frac{1}{6n}$ and $\rho_X$ is the modulus of smoothness of $X$.
\end{twr}
We shall provide the definition of the modulus of smoothness before the proof of Theorem \ref{symm}.

In the next section we consider finite dimensional Musielak-Orlicz spaces. A convex, left-continuous function $f : [0,\infty) \to [0, \infty]$, satisfying $f(0) = 0$, $\lim_{x \to \infty} f(x) = \infty $ and $f(x_0) \neq \infty$ for some $x_0>0$, is called the \emph{Young} or \emph{coordinate} function (it should be noted that in the literature their exist variations of this definition). For any collection of Young functions $f_1, f_2, \ldots, f_n$ the set
$$K = \{ (x_1, x_2, \ldots, x_n) \in \mathbb{R}^n : \sum_{i=1}^n f_i(|x_i|) \leq 1 \},$$
is easily proven to be convex, symmetric, bounded and with non-empty interior. In consequence, 
$$||x|| = \inf \{ \lambda : x \in \lambda K\}$$
defines a norm on $\mathbb{R}^n$ called the \emph{Luxemburg norm}. The space $(\mathbb{R}^n, || \cdot ||)$ is called an \emph{Musielak-Orlicz space}. Musielak-Orlicz space is called simply an \emph{Orlicz space} if $f_1=f_2=\ldots = f_n$ holds. Examples of Orlicz spaces include the $\ell^n_p$ spaces for any $1 \leq p < \infty$ with $f(t) = t^p$ being the Young function and also $\ell^n_{\infty}$ space with $f$ defined as $f(t) \equiv 0$ for $t \in [0,1]$ and $f(t) \equiv \infty$ for $t>1$. It turns out that Musielak-Orlicz spaces also satisfy Conjecture \ref{dimension}.

\begin{twr}
\label{orlicz1}
Let $X$ be an $n$-dimensional Musielak-Orlicz space. Then $e(X) \geq n+1$.
\end{twr}

An additional assumption on the coordinate functions of an Musielak-Orlicz space allows us to give yet another generalization of Theorem \ref{swanepoelvilla2}.

\begin{twr}
\label{orlicz2}
Let $X$ be an $n$-dimensional Musielak-Orlicz space whose coordinate functions $f_1, f_2, \ldots, f_n$ satisfy the condition $f_i'(0)=0$ for $i=1, 2, \ldots, n$. Then, there exist $R(X)>1$ such that $e(Y) \geq n$ for every normed space $Y$ such that $d(X, Y) \leq R(X)$.
\end{twr}
The distance $R(X)$ can be expressed in terms of coordinate functions, but in contrast to Theorem \ref{symm} this dependence is rather complicated. We provide it at the end of the proof. For more background information concerning the theory of Orlicz and Musielak-Orlicz spaces we refer the reader to \cite{musielak}.

The last considered class consists of one co-dimensional subspaces of $\ell_{\infty}^n$. It is not hard to see that, by means of approximation, in order to prove Conjecture \ref{dimension} it would be enough to prove it for every finite dimensional normed space with a polytopal unit ball. It is a folklore result that every such space occurs as a subspace of $\ell_{\infty}^n$ for some $n$. In consequence, for the purpose of establishing Conjecture \ref{dimension} it would suffice to prove it for every subspace of $\ell_{\infty}^n$ with $n \geq 1$. We shall prove it for $(n-1)$-dimensional subspaces. In fact, we can give a much better estimate on the equilateral dimension, depending on the hyperplane defining the subspace.

\begin{twr}
\label{subspace1}
Let $X = \{ (x_1, x_2, \ldots, x_n) \in \mathbb{R}^n: a_1x_1 + a_2x_2 + \ldots + a_nx_n = 0 \}$ be an $(n-1)$-dimensional subspace of the space $\ell_{\infty}^n$ and let $1 \leq k \leq n$ be an integer such that there exists a disjoint partition $\{1, 2, \ldots, n \} = A \cup B$ where $|A|=k$ and $\sum_{i \in A} |a_i| \geq \sum_{i \in B} |a_i|$. Then $e(X) \geq 2^{n-k}$.
\end{twr}

Since we can always take $k=\lceil \frac{n}{2} \rceil$ and $A$ to be set of indexes corresponding to $a_i$'s with maximal absolute value, we obtain

\begin{coll}
\label{wniosek1}
Let $X$ be an $(n-1)$-dimensional subspace of the space $\ell_{\infty}^n$. Then $e(X) \geq 2^{\lfloor \frac{n}{2} \rfloor}$.
\end{coll}

This confirms Conjecture \ref{dimension} in the considered class of spaces, as $2^{\lfloor \frac{n}{2} \rfloor} \geq n$ for $n \geq 6$, while for $n \leq 5$ the dimension of $X$ is at most $4$ and Conjecture \ref{dimension} is known to be true in this case.

Also for this class we can give an analogue of Theorems \ref{brassdekster}, \ref{swanepoelvilla}, \ref{swanepoelvilla2}.

\begin{twr}
\label{subspace2}
Let $X = \{ (x_1, x_2, \ldots, x_n) \in \mathbb{R}^n: a_1x_1 + a_2x_2 + \ldots + a_nx_n = 0 \}$ be an $(n-1)$-dimensional subspace of the space $\ell_{\infty}^n$ and let $1 \leq k \leq n$ be an integer for which there exist disjoint sets $A, B\subset \{1, 2, \ldots, n\}$ such that $|A|=k$, $|A \cup B|=n-1$ and $\sum_{i \in A} |a_i| \geq \sum_{i \in B} |a_i|$. Then $e(Y) \geq n-k$ for every normed space $Y$ such that $d(X, Y) \leq 2$. 
\end{twr}

As before we can always take $k = \lfloor \frac{n}{2} \rfloor$, which gives us

\begin{coll}
\label{wniosek2}
Let $X$ be an $(n-1)$-dimensional subspace of the space $\ell_{\infty}^n$ and let $Y$ be an $(n-1)$-dimensional space such that $d(X, Y) \leq 2$. Then $e(Y) \geq \lceil \frac{n}{2} \rceil$.
\end{coll}

The proofs of Theorems \ref{perm}, \ref{orlicz1} and \ref{subspace1} follow from an elementary continuity argument, while the proofs of Theorems \ref{symm}, \ref{orlicz2} and \ref{subspace2} are based on the approach used in the proof of Theorems \ref{swanepoelvilla} and \ref{swanepoelvilla2}, which is an extension of an idea of Brass and Dekster. Although these proofs run along similar lines, there are some adjustments necessary to fit the argument to each situation. Let us remark that Theorem \ref{symm} is the first example where the method of the Brouwer Fixed Point Theorem is applied to a quite general norm, not defined by a formula, but rather by a geometric property. This leaves the hope that such an approach can be used in an even more general setting.

For a survey on equilateral sets in finite dimensional normed spaces see \cite{swanepoel1}. For related new problems concerning maximal equilateral sets, see \cite{swanepoel2} and \cite{swanepoel3}.

\section{Permutation-invariant and symmetric spaces}
\label{symetryczne}

In this section we prove Theorems \ref{perm} and \ref{symm}. As the Conjecture \ref{dimension} is true for $n=1, 2$ we suppose that $n \geq 3$.

\begin{proof}[Proof of Theorem \ref{perm}]
Let $|| \cdot ||$ be the norm of $X$ and let $e_1, e_2, \ldots, e_n$ be the standard unit basis of $\mathbb{R}^n$. Since the space $X$ is permutation-invariant, it follows that the vectors $e_i$ form an equilateral set of the common length $c=||(1, -1, 0, \ldots, 0)||$. Moreover, for any $t \in \mathbb{R}$ and every $i=1, 2, \ldots, n$, the distance of the vector $t(e_1+e_2 + \ldots + e_n) = (t, t, \ldots, t)$ to $e_i$ is equal to
$$f(t) = ||(t-1, t, t, \ldots, t)||,$$
To obtain an $(n+1)$-th point, forming the equilateral set with the $e_i$'s, we have to find $t_0$ satisfying $f(t_0) = c$. The mapping $f$ is continuous and also $f(t) \to \infty$ as $t \to \infty$. It will thus be enough to show that $f(\frac{1}{n}) \leq c$. But this follows from the triangle inequality, as
$$f \left( \frac{1}{n} \right ) =  \left| \left | \left (-\frac{n-1}{n}, \frac{1}{n}, \ldots, \frac{1}{n} \right) \right| \right|$$ 
$$=\left| \left| \frac{1}{n} \left( -1, 1, 0 \ldots, 0 \right)  + \frac{1}{n} \left( -1, 0, 1, 0 \ldots, 0 \right ) + \ldots + \frac{1}{n} \left (-1, 0, \ldots, 0, 1 \right) \right | \right|$$
$$ \leq \frac{1}{n} \left ( ||(-1, 1, 0 \ldots, 0)|| +||(-1, 0, 1, 0 \ldots, 0)|| + \ldots + ||(-1, 0, \ldots, 1, 0)|| \right )$$
$$=\frac{n-1}{n}c < c.$$
This concludes the proof.
\end{proof}

To give a proof of the next theorem we need some preparation. In the previous section we introduced the notion of permutation-invariant and 1-unconditional spaces. We say that a finite dimensional normed space $X$ is \emph{monotone} if the norm $|| \cdot ||$ of $X$ satisfy the following condition
$$|x_i| \leq |y_i| \text{ for } i=1, 2, \ldots, n \text{ implies } ||(x_1, x_2, \ldots, x_n)|| \leq ||(y_1, y_2, \ldots, y_n)||.$$
Obviously, a monotone space is also $1$-unconditional. The converse is also true.

\begin{lem}
\label{rownowaznosc}
Let $X$ be a finite dimensional normed space. Then $X$ is monotone if and only if $X$ is $1$-unconditional.
\end{lem}
\begin{proof} See Theorem $5.5.10$ in \cite{matrix}.
\end{proof}

\begin{lem}
\label{supporting}
Let $X$ be a smooth and symmetric $n$-dimensional normed space and let $c>0$ be such that $v = (c, c, 0, \ldots, 0)$ has norm one. Then, the supporting functional of the unit sphere of $X$ at $v$ is of the form $(\frac{1}{2c}, \frac{1}{2c}, 0, \ldots, 0)$.
\end{lem}
\begin{proof}
Let $|| \cdot ||$ be a smooth and symmetric norm in $X$ and let $f$ be the supporting functional of the unit sphere at $v$. Denote by $T$ the linear mapping $T(x_1, x_2, x_3, x_4 \ldots, x_n) = (x_1, x_2, -x_3, x_4 \ldots, x_n)$. Observe that $T$ is an isometry for the norm $|| \cdot ||$ and also $T(v)=v$. In consequence, $f(T(v))=f(v) = 1$ and  $|f(T(x))| \leq ||T(x)|| = ||x||$ for every $x \in \mathbb{R}^n$. It follows that $f \circ T$ is also the supporting functional  at $v$ and hence $f= f \circ T$, since the norm is smooth. Therefore, if $f$ is of the form $f(x) = \langle x, (a_1, a_2, \ldots, a_n) \rangle$ for some real $a_i$'s, then $a_3=0$. A similar argument shows that $a_1=a_2$ and $a_4=a_5 = \ldots = a_n = 0$. Since $1=f(v)=2ca_1$ it follows that $a_1 = \frac{1}{2c}$ and the lemma is proved.
\end{proof}
Bevore proving Theorem \ref{symm} we have to define \emph{the modulus of smoothness} of a normed space $X$. This notion was introduced by Lindenstrauss in \cite{lindenstrauss} and in the equivalent form even earlier by Day in \cite{day}. It is the function $\rho_X: (0, \infty) \to \mathbb{R}$ defined as
$$\rho_X(t) = \sup \left \{ \frac{||x+ty|| + ||x-ty||}{2} -1 : ||x||, ||y|| \leq 1 \right \}.$$
Finite dimensional normed space $X$ is smooth if and only if $\lim_{t \to 0^{+}} \frac{\rho_X(t)}{t} = 0$ (see Fact $9.7$ in \cite{modul}). Therefore, if $X$ is smooth than there exists $\varepsilon_0>0$ such that $\frac{\rho_X(\varepsilon_0)}{\varepsilon_0} \leq \frac{1}{6n}$.

Now we are ready to give the proof of Theorem \ref{symm}.

\begin{proof}[Proof of Theorem \ref{symm}]
We shall follow the approach of Swanepoel and Villa used in the proof of Theorem \ref{swanepoelvilla2}. Suppose that $R=R(X)$ is defined as in the statement of the theorem. Let $X = (\mathbb{R}^n, || \cdot ||)$ and $Y=(\mathbb{R}^n, || \cdot ||_Y)$.  Since $e(Y)=e(Y')$ for every space $Y'$ linearly isometric to $Y$, without loss of generality we may assume that
$$||x||_Y \leq ||x|| \leq R ||x||_Y$$
for every $x \in \mathbb{R}^n$. After an appropriate rescaling we can further suppose that $||e_i||=1$ for every $i=1, 2, \ldots, n$. We fix $\beta, \gamma>0$ and denote by $I$ the set of pairs $\{(i, j): 1 \leq i < j \leq n \},$ consisting of $N=\frac{n(n-1)}{2}$ elements.
For $\varepsilon = (\varepsilon_{i, j})_{(i, j) \in I} \in [0, \beta]^N$, let
$$p_1(\varepsilon) = (-\gamma, 0, \ldots, 0),$$
$$p_j(\varepsilon) = (\varepsilon_{1, j}, \ldots, \varepsilon_{j-1, j}, -\gamma, 0, \ldots, 0), \quad 2 \leq j \leq n-1,$$
$$p_n(\varepsilon) = (\varepsilon_{1, n}, \ldots, \varepsilon_{n-1, n}, -\gamma).$$
Define $\varphi:[0, \beta]^N \to \mathbb{R}^N $ by 
$$\varphi_{i,j}(\varepsilon) =
1+\varepsilon_{i,j}-||{p_i(\varepsilon)-p_j(\varepsilon)}||_Y\\ \text{ for }  1\le i<j\le n.$$
Notice that the $k$-th coordinate of the vector $p_j-p_i$ (where $1 \leq i < j \leq n$) is equal to
$$
\begin{cases}
\varepsilon_{k, j} - \varepsilon_{k, i} &\mbox{ for } 1 \leq k < i\\ 
\varepsilon_{i, j} + \gamma &\mbox{ for } k=i\\
\varepsilon_{k, j} &\mbox{ for } i < k < j \\
\gamma &\mbox{ for } k=j\\ 
0 &\mbox{ for } k>j.$$
\end{cases}
$$
The norm $|| \cdot ||$ is $1$-unconditional and therefore also monotone by Lemma \ref{rownowaznosc}. As a consequence
$$||(\gamma+\varepsilon_{i, j}, \gamma, 0, 0 \ldots 0)|| \leq ||p_j - p_i|| \leq ||(\gamma+\varepsilon_{i, j}, \gamma, \beta, \beta, \ldots, \beta)||.$$
To apply the Brouwer Fixed Point Theorem to $\varphi$ we have to choose parameters $\beta, \gamma$ in such a way that the image of $\varphi$ is contained in $[0, \beta]^N$. For this purpose we estimate
$$\varphi_{i,j}(\varepsilon) =
1+\varepsilon_{i,j}-||{p_i(\varepsilon)-p_j(\varepsilon)}||_Y$$
$$\leq 1+\varepsilon_{i,j} - R^{-1}||{p_i(\varepsilon)-p_j(\varepsilon)}|| \leq 1+\varepsilon_{i,j} - R^{-1}||(\gamma+\varepsilon_{i, j}, \gamma, 0, \ldots, 0)||.$$
Similarly
$$\varphi_{i,j}(\varepsilon) =
1+\varepsilon_{i,j}-||{p_i(\varepsilon)-p_j(\varepsilon)}||_Y \geq 1+\varepsilon_{i,j} - ||{p_i(\varepsilon)-p_j(\varepsilon)}||$$
$$\geq 1+\varepsilon_{i,j} - ||(\gamma+\varepsilon_{i, j}, \gamma, \beta, \ldots, \beta)||.$$
We notice that smoothness of the norm provides differentiability of the function $h(\varepsilon) = 1 + \varepsilon -  R^{-1}||(\gamma+\varepsilon, \gamma, 0, \ldots, 0)||$ for $\varepsilon \geq 0$. Moreover, it is not hard to see that $h(\varepsilon)$ is strictly increasing. In fact, the triangle inequality yields
$$\frac{||(\gamma+\varepsilon + t, \gamma, 0, \ldots, 0)|| - ||(\gamma+\varepsilon, \gamma, 0, \ldots, 0)||}{t} \leq \frac{||(t, 0, \ldots, 0))||}{t}$$
$$=||(1, 0, \ldots, 0)|| = 1,$$
for any $t>0$. Since $R>1$, an analogous argument for $t<0$ proves that $h'(\varepsilon)>0$.

In the same way we can show that $1+\varepsilon - ||(\gamma+\varepsilon, \gamma, \beta, \ldots, \beta)||$ is increasing. Therefore, we have to choose $\beta$ and $\gamma$ satisfying the inequalities
$$||(\gamma+\beta, \gamma, 0 \ldots, 0)|| \geq 1 + \frac{\varepsilon_0}{6n} \: \text { and } \: ||(\gamma, \gamma, \beta, \ldots, \beta)|| \leq 1.$$
Then $R = ||(\gamma+\beta, \gamma, 0 \ldots, 0)|| \geq 1 + \frac{\varepsilon_0}{6n}$ would satisfy conditions of the theorem.

Let $c$ be a positive real number for which $||(c, c, 0, \ldots, 0)||=1$. As $|| \cdot ||$ is monotone and $||(1, 0, 0, \ldots, 0)|| = 1$ we clearly have $c \leq 1$. Let us take $\gamma = c-\varepsilon$, $\beta = 3 \varepsilon$, where $\varepsilon = \frac{\varepsilon_0}{3n}$. According to Lemma \ref{supporting} we know that $f(x) = \langle x, (\frac{1}{2c}, \frac{1}{2c}, 0, \ldots, 0) \rangle$ is the supporting functional of the unit sphere at $x_0=(c, c, 0, \ldots, 0)$. Consider the vector $v=(-1, -1, 3, \ldots, 3) \in \mathbb{R}^n$. As
$$(\gamma, \gamma, \beta, \ldots, \beta)  = (c-\varepsilon, c-\varepsilon, 3\varepsilon, \ldots, 3\varepsilon)=x_0 + \varepsilon v$$
we have to show that $||x_0 + \varepsilon v|| \leq 1.$

From the triangle inequality and assumption $||e_i||=1$, for $1 \leq i \leq n$, it follows that
$$||v|| \leq 3n-4 < 3n.$$
By the definition of $\rho_X$ we have
$$||x_0 + \varepsilon v| + ||x_0 - \varepsilon v||=\left| \left | x_0 + \varepsilon_0 \frac{v}{3n} \right | \right | + \left | \left |x_0 - \varepsilon_0 \frac{v}{3n} \right| \right| \leq 2\rho_X(\varepsilon_0)+2,$$
or equivalently
$$\frac{||x_0 + \varepsilon v||}{\varepsilon} + \frac{||x_0 - \varepsilon v||-1}{\varepsilon} \leq \frac{2p_X(\varepsilon_0)}{\varepsilon} + \frac{1}{\varepsilon}.$$
Since $||x_0 - \varepsilon v||\geq f(x_0 - \varepsilon v) = 1 + \frac{\varepsilon}{c}$ we have
$$\frac{||x_0 - \varepsilon v||-1}{\varepsilon} \geq \frac{1}{c} \geq 1.$$
From the choice of $\varepsilon_0$ it follows that
$$\frac{2\rho_X(\varepsilon_0)}{\varepsilon} = \frac{6n\rho_X(\varepsilon_0)}{\varepsilon_0} \leq 1.$$
Combining this estimations gives us
$$\frac{||x_0 + \varepsilon v||}{\varepsilon} + 1 \leq \frac{||x_0 + \varepsilon v||}{\varepsilon} + \frac{||x_0 - \varepsilon v||-1}{\varepsilon} \leq \frac{2\rho_X(\varepsilon_0)}{\varepsilon} + \frac{1}{\varepsilon} \leq 1 + \frac{1}{\varepsilon},$$
and hence $||x_0 + \varepsilon v|| \leq 1.$

To obtain the desired conclusion observe that for $w=(2, -1, 0, \ldots, 0) \in \mathbb{R}^n$ we have $f(w)=\frac{1}{2c}$ and therefore
$$R = ||c+2\varepsilon, c-\varepsilon, 0, \ldots, 0)||  = ||x_0 + \varepsilon w|| \geq f(x_0 + \varepsilon w) =  1 + \frac{\varepsilon}{2c} = 1 + \frac{\varepsilon_0}{6nc} \geq 1 + \frac{\varepsilon_0}{6n}.$$
This completes the proof.
\end{proof}

\section{Musielak-Orlicz spaces}
\label{orlicze}
In this section we use symbols $f^{-}(a)$ and $f^{+}(a)$ for the left and right derivative respectively. Similarly like before we assume that $n \geq 3$.

\begin{lem}
\label{apr}
If $e(X) \geq n+1$ for any $n$-dimensional Musielak-Orlicz space with strictly increasing and finite valued coordinate functions $f_1, f_2, \ldots, f_n$, then $e(X) \geq n+1$ for any $n$-dimensional Musielak-Orlicz space.
\end{lem}
\begin{proof}
The proof goes by a standard approximation argument. Suppose that Conjecture \ref{dimension} holds for any $n$-dimensional Musielak-Orlicz space satisfying the condition from the lemma, and let $X$ be any $n$-dimensional Musielak-Orlicz space with the coordinate functions $f_1, f_2, \ldots, f_n$. Let $a_i = \sup \{ x \in \mathbb{R}: f_i(x) = 0 \}$ and $b_i = \sup \{ x \in \mathbb{R}: f_i(x) < \infty \}$ for $i=1, 2, \ldots, n$. For any integer $k > 0 $ let us define $f_{i, k}(x) = \frac{k-1}{k}f(x) + \frac{x}{k}$ for $x \in [0, b_i]$ and $i=1, 2, \ldots, n$. If $b_i=\infty$, that is, the function $f_i$ takes only finite values, then this already defines $f_{i, k}$ on $[0, +\infty)$. In the other case, let $f_{i,k}(x) = \left ( \frac{f_{i,k}(b_i)}{b_i} + kf^{-}_{i,k}(b_i) \right )x - kf^{-}_{i,k}(b_i)b_i$ for $x \in (b_i, \infty)$. It is not hard to see that each $f_i,k$ is strictly increasing and finite valued Young function. Moreover, if we denote by $|| \cdot ||_k$ the Luxemburg norm associated to $f_{1, k}, f_{2, k}, \ldots, f_{n, k}$ for $k>0$, then it is straightforward to check that
$$\lim_{k \to \infty} ||x||_k = ||x||,$$
for every $x \in \mathbb{R}^n$.

By assumption, we can find a set of $n+1$ equidistant points in each norm $|| \cdot ||_k$. Let $0, p_{1, k}, p_{2, k}, \ldots, p_{n, k}$ be an equilateral set of a common distance $1$ in the norm $|| \cdot ||_k$. The sequence $(p_{1,k})_{k>0}$ is bounded and therefore it has a convergent subsequence to some $p_1 \in \mathbb{R}^n$. After repeating this argument $n$ times we get $n$ points $p_1, p_2, \ldots, p_n$ which form, along with $0$, an equilateral set in the norm $|| \cdot ||$. This concludes the proof of the lemma.
\end{proof}

\begin{proof}[Proof of Theorem \ref{orlicz1}.]
Let $f_1, f_2, \ldots, f_n$ be the coordinate functions of $X$, that is the norm $|| \cdot ||$ of $X$ is given by
$$||(x_1, x_2, \ldots, x_n)|| = \inf \left \{ r>0: \sum_{i=1}^{n} f_i \left ( \frac{|x_i|}{r} \right ) \leq 1 \right \}.$$
By the preceeding lemma, we can assume that all functions $f_i$'s are strictly increasing with finite values. It is easy to see that in such a setting, the unit sphere of $X$ consists of exactly these vectors $x=(x_1, x_2, \ldots, x_n)$ for which $\sum_{i=1}^{n} f_i \left (|x_i| \right )= 1.$

Since all $f_i$ are continuous mappings with image equal to $[0, +\infty)$, for all $i=1, 2, \ldots, n$ there exists $c_i>0$ such that $f_i(c_i)=\frac{1}{2}$. One easily verifies that $c_ie_i$ are equidistant with the common distance equal to $1$. To prove the theorem, it will thus be sufficient to find a point $t=(t_1, t_2, \ldots, t_n)$ at distance $1$ to every $c_ie_i$.

For $i=1, 2, \ldots, n$ consider the function $g_i: [0, c_i] \to \mathbb{R}$ defined as $g_i(x) = f_i(c_i - x) - f_i(x)$. Since each $f_i$ is strictly increasing it easily follows that each $g_i$ is a strictly decreasing, continuous function with the image $[-\frac{1}{2}, \frac{1}{2}]$. Hence, each $g_i$ has continuous inverse $g^{-1}_i: [-\frac{1}{2}, \frac{1}{2}] \to [0, c_i]$. Fix $t_1 \in [0, c_1]$ and take $t_i = g^{-1}_i(g_1(t_1))$ for $2 \leq i \leq n$, so that $g_i(t_i) = g_1(t_1)$. Then
$$f_i(c_i - t_i) + \sum_{1 \leq j \leq n, j \neq i} f_j(t_j) = g_i(t_i) + \sum_{1 \leq j \leq n} f_j(t_j)$$
$$= g_1(t_1) + \sum_{1 \leq j \leq n} f_j(t_j) = f_1(c_1 - t_1) + \sum_{2 \leq j \leq n} f_j(t_j),$$
for any $i=1, 2, \ldots, n$. This shows that $t=(t_1, t_2, \ldots, t_n)$ is equidistant to every $c_ie_i$. We have to choose a $t_1 \in [0, c_1]$ such that the common distance is equal to $1$. For this purpose, let us define
$$h(t_1) = g_1(t_1) + \sum_{1 \leq j \leq n} f_j(g^{-1}_j(g_1(t_1))) = g_1(t_1) + \sum_{1 \leq j \leq n} f_j(t_j).$$
It is clear that $h$ is continuous, $h(0) = \frac{1}{2}  <  1$ and $h(c_1) = -\frac{1}{2} + \frac{n}{2} \geq 1$. Hence $h(t_1)=1$ for some $t_1 \in [0, c_1]$ and the proof is completed.

\end{proof}

\begin{proof}[Proof of Theorem \ref{orlicz2}.]
We follow a similar idea to the proof of Theorem \ref{symm}. The constant $R=R(X)$ shall be defined later on in the proof. Let $X = (\mathbb{R}^n, || \cdot ||)$ and $Y=(\mathbb{R}^n, || \cdot ||_Y)$. As in the proof of Theorem \ref{symm} we can assume that
$$||x||_Y \leq ||x|| \leq R ||x||_Y$$
for every $x \in \mathbb{R}^n$. After an appropriate rescaling we can further suppose that $||e_i||<1$ for every $i=1, 2, \ldots, n$. Moreover, let us also assume that $f_1, f_2, \ldots, f_m$ (where $0 \leq m \leq n$) are exactly these coordinate functions for which there does not exist $c_i>0$ such that $f_i(c_i) = \frac{1}{2}$. It is clear that these coordinate functions have to attain the value $\infty$, so let $c_i = \sup \{ x \geq 0: f_i(x) < \infty \}$ for $i \leq m$ (in particular $f_i(c_i) < \frac{1}{2}$ for $0 \leq i \leq m$).

We fix $\beta, \gamma_1, \gamma_2, \ldots, \gamma_n>0$ and denote by $I$ the set of pairs $\{(i, j): 1 \leq i < j \leq n \},$ consisting of $N=\frac{n(n-1)}{2}$ elements.
For $\varepsilon = (\varepsilon_{i, j})_{(i, j) \in I} \in [0, \beta]^N$, let
$$p_1(\varepsilon) = (-\gamma_1, 0, \ldots, 0),$$
$$p_j(\varepsilon) = (\varepsilon_{1, j}, \ldots, \varepsilon_{j-1, j}, -\gamma_j, 0, \ldots, 0), \quad 2 \leq j \leq n-1,$$
$$p_n(\varepsilon) = (\varepsilon_{1, n}, \ldots, \varepsilon_{n-1, n}, -\gamma_n).$$
Define $\varphi:[0, \beta]^N \to \mathbb{R}^N $ by
$$\varphi_{i,j}(\varepsilon) =
1+\varepsilon_{i,j}-||{p_i(\varepsilon)-p_j(\varepsilon)}||_Y\\ \text{ for }  1\le i<j\le n.$$
Notice that the $k$-th coordinate of the vector $p_j-p_i$ (where $1 \leq i < j \leq n$) is equal to
$$
\begin{cases}
\varepsilon_{k, j} - \varepsilon_{k, i} &\mbox{ for } 1 \leq k < i\\ 
\varepsilon_{i, j} + \gamma_i &\mbox{ for } k=i\\
\varepsilon_{k, j} &\mbox{ for } i < k < j \\
\gamma_j &\mbox{ for } k=j\\ 
0 &\mbox{ for } k>j.$$
\end{cases}
$$
Monotonicity of the norm $|| \cdot ||$ (Lemma \ref{rownowaznosc}) yields
$$||(0, \ldots, 0, \gamma_i+\varepsilon_{i, j}, 0, \ldots, 0, \gamma_j, 0 \ldots 0)|| \leq ||p_j - p_i|| \leq $$
$$||(\beta, \ldots, \beta, \gamma_i+\varepsilon_{i, j}, \beta, \ldots, \beta, \gamma_j, \beta, \ldots, \beta)||.$$
We want to choose parameters $\beta, \gamma_1, \ldots, \gamma_n$ in such a way that the image of $\varphi$ is contained in $[0, \beta]^N$ and then apply the Brouwer Fixed Point Theorem. For this purpose, we estimate
$$\varphi_{i,j}(\varepsilon) =
1+\varepsilon_{i,j}-||{p_i(\varepsilon)-p_j(\varepsilon)}||_Y \leq 1+\varepsilon_{i,j} - R^{-1}||{p_i(\varepsilon)-p_j(\varepsilon)}||$$
$$\leq 1+\varepsilon_{i,j} - R^{-1}||(0, \ldots, 0, \gamma_i+\varepsilon_{i, j}, 0, \ldots, 0, \gamma_j, 0 \ldots 0)||,$$
$$\varphi_{i,j}(\varepsilon) =
1+\varepsilon_{i,j}-||{p_i(\varepsilon)-p_j(\varepsilon)}||_Y \geq 1+\varepsilon_{i,j} - ||{p_i(\varepsilon)-p_j(\varepsilon)}||$$
$$\geq 1+\varepsilon_{i,j} - ||(\beta, \ldots, \beta, \gamma_i+\varepsilon_{i, j}, \beta, \ldots, \beta, \gamma_j, \beta, \ldots, \beta)||.$$
The function $h(\varepsilon) =  1 + \varepsilon -  R^{-1} ||(0, \ldots, 0, \gamma_i+\varepsilon, 0, \ldots, 0, \gamma_j, 0 \ldots 0)||$ is not necessarily differentiable for $\varepsilon \geq 0$, in contrast to the proof of Theorem \ref{symm}. Nevertheless, it is straightforward to check that it is concave and therefore it has left and right derivative in every $\varepsilon>0$. Taking into account that $||e_i|| < 1$ for $i=1, 2, \ldots, n$ we can thus repeat the argument used in the proof of Theorem \ref{symm}, to prove that it is increasing for $\varepsilon \geq 0$. Analogously, $1+\varepsilon -  ||(\beta, \ldots, \beta, \gamma_i+\varepsilon, \beta, \ldots, \beta, \gamma_j, \beta, \ldots, \beta)||$ is also increasing. Therefore, in order to have $0 \leq \varphi_{i, j}(\varepsilon) \leq \beta$ we have to choose $\beta$ and $\gamma_1, \gamma_2, \ldots, \gamma_n$ satisfying the system of inequalities
$$||(0, \ldots, 0, \gamma_i+\beta, 0, \ldots, 0, \gamma_j, 0 \ldots 0)|| > 1,$$
$$||(\beta, \ldots, \beta, \gamma_i, \beta, \ldots, \beta, \gamma_j, \beta, \ldots, \beta)|| \leq 1,$$
for all $1 \leq i < j \leq n$. Then we could take 
$$R=\min_{(i, j) \in I}  ||(0, \ldots, 0, \gamma_i+\beta, 0, \ldots, 0, \gamma_j, 0 \ldots 0)||.$$

The above system of inequalities is clearly equivalent to
$$f_i(\gamma_i + \beta) + f_j(\gamma_j) > 1,$$
$$f_i(\gamma_i) + f_j(\gamma_j) + \sum_{k \neq i,j} f_k(\beta) \leq 1,$$
for all $(i, j) \in I$. Let us define $\gamma_i = c_i - \varepsilon$ and $\beta=(K+1) \varepsilon$, where 
$$K > \max \left \{ \frac{f_j^{-}(c_j)}{f_i^{+}(c_i)}: m < i < j \leq n \right \}.$$ We shall prove that the desired conditions are satisfied for sufficiently small $\varepsilon>0$. 

For this purpose let us consider two cases:\\
\textbf{1}. $i > m$ (and consequently $j > m$). Let 
$$g(\varepsilon) = f_i(c_i + K\varepsilon) + f_j(c_j - \varepsilon)$$
for $\varepsilon \in [0, c_j]$. As $g$ is convex we have
$$\frac{g(\varepsilon) - g(0)}{\varepsilon} \geq g^{+}(0) = Kf^{+}_i(c_i) - f^{-}_j(c_j)>0.$$
for $\varepsilon>0$. Since $g(0)=1$, it follows that $g(\varepsilon)>1$ for $\varepsilon>0$.

Now let us consider
$$h(\varepsilon) = f_i(c_i - \varepsilon) + f_j(c_j - \varepsilon) + \sum_{k \neq i,j} f_k((K+1)\varepsilon).$$
It is clear that $f^{-}_i(c_i), f^{-}_j(c_j)>0$. As $f'_k(0)=0$ for every $1 \leq k \leq n$ we can therefore choose $\varepsilon$ such that $\sum_{k=1}^{n} \frac{f_k((K+1)\varepsilon)}{\varepsilon} \leq f^{-}_i(c_i) + f^{-}_j(c_j)$. Then
$$\frac{h(\varepsilon)}{\varepsilon} = \frac{f_i(c_i - \varepsilon)}{\varepsilon} + \frac{f_j(c_j - \varepsilon)}{\varepsilon} + \frac{\sum_{k \neq i,j} f_k((K+1)\varepsilon)}{\varepsilon}$$ $$=\frac{f_i(c_i - \varepsilon)-f_i(c_i)}{\varepsilon} + \frac{f_j(c_j - \varepsilon)-f_j(c_j)}{\varepsilon} + \frac{\sum_{k \neq i,j} f_k((K+1)\varepsilon)}{\varepsilon}+\frac{f_i(c_i) + f_j(c_j)}{\varepsilon}$$
$$\leq -(f^{-}_i(c_i) + f_j^{-}(c_j)) + (f^{-}_i(c_i) + f_j^{-}(c_j)) + \frac{1}{\varepsilon} = \frac{1}{\varepsilon}.$$
and hence for such chosen $\varepsilon$ we have $h(\varepsilon) \leq 1$. This proves that for $i>m$ we can choose $\varepsilon>0$ satisfying the system of inequalities above.
\\
\textbf{2.} $i \leq m$. Then $f_i(\gamma_i + \beta) = f_i(c_i + K \varepsilon)= \infty$ for $\varepsilon>0$, so the first inequality is satisfied. For the second one, observe that $f_i(c_i - \varepsilon) + f_j(c_j - \varepsilon) < f_i(c_i) + f_j(c_j) < 1$. Since $f_k(0)=0$ and $f_k$ are continuous in $0$ for all $k$, it is clear that for sufficiently small $\varepsilon>0$ the second inequality will also hold. More precisely, it is enough to take $\varepsilon>0$ satisfying $\sum_{k=1}^{n} f_k((K+1)\varepsilon) \leq 1 - f_i(c_i) - f_j(c_j).$

This gives the desired conclusion. To summarize we shall give more comprehensive expression for $R(X)$ in terms of the coordinate functions.
$$R(X) \geq \min_{(i, j) \in I}  ||(0, \ldots, 0, c_i + K\varepsilon_0, 0, \ldots, 0, c_j - \varepsilon_0, 0 \ldots 0)||$$
$$=\min_{(i, j) \in I} \inf \left \{ \lambda: f_i \left ( \frac{c_i + K\varepsilon_0}{\lambda} \right ) + f_j \left ( \frac{c_j - \varepsilon_0}{\lambda} \right )  \leq 1 \right \},$$
where
$$K > \max \left \{ \frac{f_j^{-}(c_j)}{f_i^{+}(c_i)}: m < i < j \leq n \right \}$$
and $\varepsilon_0>0$ satisfies conditions
$$\sum_{k=1}^{n} \frac{f_k((K+1)\varepsilon_0)}{\varepsilon_0} \leq \min \{ f^{-}_i(c_i) + f^{-}_j(c_j) : m < i < j \leq n \},$$
$$\sum_{k=1}^{n} f_k((K+1)\varepsilon_0) \leq \min \{ 1 - f_i(c_i) - f_j(c_j) : 1 \leq i \leq m, \: i < j \leq n \}.$$



\end{proof}

\section{Subspaces of $\ell_{\infty}^n$ of codimension one}
\label{podprz}
In this section we shall use the following notation: if $v_i \in \mathbb{R}^{n_i}$ for $i=1, 2, \ldots, k$ then by $(v_1, v_2, \ldots, v_k)$ we mean a standard concatenation in  $\mathbb{R}^{n_1+n_2 + \ldots + n_k}$. We will also use the symbol $\mathbf{0}_n$ to distinguish the zero vector of the space $\mathbb{R}^n$.

\begin{proof}[Proof of Theorem \ref{subspace1}]
Note that every mapping of the form 
$$T: \ell^n_{\infty} \ni (x_1, x_2, \ldots, x_n) \to (c_1 x_1, c_2 x_2, \ldots, c_n x_n) \in \ell^n_{\infty},$$
where $c_i \in \{-1, 1\}$ for $i=1, 2, \ldots, n$, is a linear isometry. Since $e(X')=e(X)$ for every space $X'$ linearly isometric to $X$, we can therefore assume that the coefficients $a_i$ are nonnegative. Let us suppose that $0 \leq a_1 \leq a_2 \leq \ldots \leq a_n$. In particular $\sum_{i=1}^{n-k} a_i \leq \sum_{i=n-k+1}^{n} a_i$. Consider the mappings $h: \mathbb{R}^{n-k} \to \mathbb{R}$, $H: \mathbb{R}^{n-k} \to \mathbb{R}^{k}$ defined as 
$$h(x) =\frac{\sum_{i=1}^{n-k} a_ix_i}{\sum_{i=n-k+1}^{n} a_i}, \quad H(x) = (-h(x), -h(x), \ldots, -h(x)).$$
From the assumption on the coefficients $a_i$'s it follows that $|h(x)| \leq ||x||_{\infty}$ for every $x \in \mathbb{R}^{n-k}$. Note also that $(x, H(x)) \in X$ for every $x \in \mathbb{R}^{n-k}$. Let
$$ S = \left \{  \left(c, H(c) \right ) : c \in \{1, -1 \}^{n-k} \right \} \subset X.$$
Every two distinct elements of $S$ differ in at least one of the first $n-k$ coordinates, and the absolute value of every other coordinate is bounded by $1$ (since $||H(c)||_{\infty} \leq ||c||_{\infty}=1$ for $c \in \{-1, 1\}^{n-k}$). It follows that every two elements of $S$ are at distance $2$ in the $\ell_{\infty}$ norm. We have thus obtained an equilateral set in $X$ of cardinality $2^{n-k}$ and the result follows.

\end{proof}


\begin{proof}[Proof of Theorem \ref{subspace2}]
Similarly to the proof of Theorem \ref{subspace1} we can suppose that $0 \leq a_1 \leq a_2 \leq \ldots \leq a_n$. In particular $\sum_{i=1}^{n-k-1} a_i \leq \sum_{i=n-k+1}^{n} a_i$. Let us further assume that linear structure of $Y$ is identified with the linear structure of $X$ and the norm $|| \cdot ||$ of $Y$ satisfies
$$||x|| \leq ||x||_{\infty} \leq 2||x||,$$
for every $x \in X$ (which is again possible by passing to a suitable space $Y'$ linearly isometric to $Y$).  Denote by $I$ the set of pairs $\{(i, j): 1 \leq i < j \leq n-k \},$ consisting of $N=\frac{(n-k)(n-k-1)}{2}$ elements.

We shall introduce mappings $p_j:[0,1]^N \to \mathbb{R}$ for $j=1, 2, \ldots, n-k$. For $\varepsilon = (\varepsilon_{i, j})_{(i, j) \in I} \in [0, 1]^N$ let us define

$$p_j(\varepsilon)=
\begin{cases}
\left (-1, \mathbf{0}_{n-k-2}, b_1, \mathbf{0}_{k} \right) &\mbox{for } j=1,\\
\left (\varepsilon_j, -1, \mathbf{0}_{n-k-j-1}, b_j, H(\varepsilon_j, \mathbf{0}_{n-k-j}) \right) &\mbox{for } 2 \leq j \leq n-k-1,\\
(\varepsilon_{n-k}, 0, H(\varepsilon_{n-k})) &\mbox{for } j=n-k,\\
\end{cases}$$

where 
$$\varepsilon_j=( \varepsilon_{1, j}, \varepsilon_{2, j}, \ldots, \varepsilon_{j-1, j}) \in \mathbb{R}^{j-1} \qquad (2 \leq j \leq n-k),$$
$$b_j= 
\begin{cases}
 \frac{a_j}{a_{n-k}} &\mbox{if } a_{n-k} \neq 0\\
0 &\mbox{if } a_{n-k}=0
\end{cases}
\qquad (1 \leq j \leq n-k),$$
$$h:\mathbb{R}^{n-k-1} \to \mathbb{R}, \quad h(x)= \frac{\sum_{i=1}^{n-k-1} a_ix_i}{\sum_{i=n-k+1}^{n} a_i},$$
$$H:\mathbb{R}^{n-k-1} \to \mathbb{R}^k, \quad H(x) = (-h(x), -h(x), \ldots, -h(x)) \in \mathbb{R}^k.$$

From the assumption on the coefficients $a_i$'s it follows that $0 \leq h(x) \leq 1$ for $x \in [0, 1]^{n-k-1}$. Moreover $p_j(\varepsilon) \in X$ for every $j=1, 2, \ldots, n-k$. Indeed, suppose that $p_j(\varepsilon)=(q_1, q_2, \ldots, q_n)$. If $a_{n-k}=0$ then $a_j=0$ and $b_j=0$.  If $a_{n-k} \neq 0$ then $b_j = \frac{a_j}{a_{n-k}}$. In each case
$$a_j q_j + a_{n-k} q_{n-k} =  -a_j + a_{n-k}b_j = -a_j + a_j = 0.$$
Furthermore
$$\sum_{ i \neq j, i \neq n-k } a_i q_i= a_1 \varepsilon_{1, j} + \ldots + a_{j-1} \varepsilon_{j-1, j} - (a_{n-k+1} + a_{n-k+2}  + \ldots + a_n)h(\varepsilon_j, \mathbf{0}_{n-k-j})$$
$$=a_1 \varepsilon_{1, j} + \ldots + a_{j-1} \varepsilon_{j-1, j} - (a_1 \varepsilon_{1, j} + \ldots + a_{j-1} \varepsilon_{j-1, j}) = 0.$$

Now we claim that $||p_j (\varepsilon)- p_l(\varepsilon)||_{\infty} = 1 + \varepsilon_{j, l}$ for $1 \leq  j < l \leq m$. In fact, let $p_j(\varepsilon)=(q_1, q_2, \ldots, q_n)$ and $p_l(\varepsilon)=(r_1, r_2, \ldots, r_n)$. Then $|q_j - r_j|=1+\varepsilon_{j,l}$ and therefore $||p_j (\varepsilon)- p_l(\varepsilon)||_{\infty} \geq 1 + \varepsilon_{j, l}$. On the other hand, we have $q_i, r_i \in [0, 1]$ for $1 \leq i \leq n-k, i \neq j$ and $q_i, r_i \in [-1, 0]$ for $n-k < i \leq n$. We conclude that that $|q_i-r_i| \leq 1$ for every $i \neq j$ and our claim follows.

Now we can proceed in the same way as Swanepoel and Villa did in \cite{swanepoel2}. Define $\varphi:[0, 1]^N \to \mathbb{R}^N $ by 
$$\varphi_{i,j}(\varepsilon) = 1+\varepsilon_{i,j}-||{p_i(\varepsilon)-p_j(\varepsilon)}||,$$ 
for $ 1 \leq  i < j \leq n-k$. We claim that the image of $\varphi$ is contained in $[0, 1]^N$. Indeed, by using the estimates on the norm we have
$$\varphi_{i,j}(\varepsilon) =1+\varepsilon_{i,j}-||{p_i(\varepsilon)-p_j(\varepsilon)}|| \geq 1+\varepsilon_{i,j}-||{p_i(\varepsilon)-p_j(\varepsilon)}||_{\infty}=0,$$
$$\varphi_{i,j}(\varepsilon) =1+\varepsilon_{i,j}-||{p_i(\varepsilon)-p_j(\varepsilon)}|| \leq 1+\varepsilon_{i,j}-\frac{1}{2}||{p_i(\varepsilon)-p_j(\varepsilon)}||_{\infty} = \frac{1}{2}(1+\varepsilon_{i, j}) \leq 1.$$
 It is clear that the mappings $p_j(\varepsilon)$ are continuous and therefore $\varphi$ is also a continuous mapping. Thus, the Brouwer Fixed Point Theorem gives the existence of a point $\varepsilon' \in [0, 1]^N $ such that $\varphi(\varepsilon') = \varepsilon'$ and consequently  $||p_i(\varepsilon') - p_j(\varepsilon')|| = 1$ for $1 \leq i < j \leq m$. This gives an equilateral set of $n-k$ elements and the result follows.

\end{proof}

\section{concluding remarks}

In the preceeding sections we gave some evidence on Conjecture \ref{dimension} by confirming it in some well-known classes of normed spaces. We close this paper by offering some naturally arising questions and problems for further research.

In section \ref{symetryczne} we proved Theorem \ref{dimension} for the spaces, whose groups of isometries contain the permutation of coordinate variables. It is reasonable to ask about some other natural condition on the group of isometries that would be enough to obtain an equilateral set of large cardinality. The class of $1$-unconditional spaces comes to mind.

\begin{problem}
Prove that for any $n$-dimensional $1$-unconditional space $X$ we have $e(X) \geq n+1$.
\end{problem}

We remark that in fact there are some open problems of convex geometry which can be established under the assumption of $1$-unconditionallity. A good example is a Mahler conjecture on the volume of convex bodies (see \cite{mahler}).

In the theory of classical Orlicz spaces there is also another norm extensively studied. It is called the \emph{Amemiya-Orlicz norm} and in the finite dimensional setting it can be defined by the formula
$$||x|| = \inf_{\lambda > 0} \frac{1}{\lambda} \left ( \sum_{i=1}^{n} f(\lambda |x_i|) + 1 \right ),$$
where $f_1, f_2, \ldots, f_n$ are Young functions. It would be interesting to know if Conjecture \ref{dimension} could be deduced for the Amemiya-Orlicz norms.

\begin{problem}
Prove that for any $n$-dimensional space $X$, endowed with the an Amemiya-Orlicz norm, we have $e(X) \geq n+1$.
\end{problem}

Note that Theorem \ref{swanepoelvilla2} of Swanepoel and Villa concerning the $\ell^n_p$ spaces does not include the case of $\ell_1^n$ space. Moreover, it also clear that $\ell_1^n$ fails to satisfy the conditions of our generalizations (Theorems \ref{symm}, \ref{orlicz2}). One could ask about an analogue of Theorems \ref{swanepoelvilla2}, \ref{symm}, \ref{orlicz2} in the case of this familiar space.

\begin{problem}
Prove that there exists $R>1$ such that $e(X) \geq n$ for any $n$-dimensional normed space $X$ satisfying $d(\ell_1^n, X) \leq R$.
\end{problem}

As mentioned in the introduction, to prove Conjecture \ref{dimension} it would be sufficient to prove it for every subspace of the space $\ell_{\infty}^n$ for all $n$. In Section \ref{podprz} we handled the case of subspaces of codimension $1$. Already the case of subspaces of codimension $2$ seems to be much more delicate.

\begin{problem}
Prove that for any $(n-2)$-dimensional subspace $X$ of the space $\ell_{\infty}^n$ we have $e(X) \geq n-1$.
\end{problem}

There are many other naturally arising classes of normed spaces and convex bodies in which the problems of equilateral sets could be further investigated.

\section*{Acknowledgements}
I thank the anonymous referees for very careful proofreading and suggestions which significantly improved quality of the paper.

\bibliographystyle{amsplain}

\end{document}